\documentclass[10pt]{amsart}

\usepackage{amsmath, amssymb,amsmath, graphicx, xypic}
\usepackage[small]{caption}
\usepackage{txfonts}

\numberwithin{equation}{section}

\def\N{\mathbb{N}} 
\def\Z{\mathbb{Z}}
\def\ZH{\mathbb{Z}H}
\def\ZG{\mathbb{Z}G}

\newtheorem{theorem}{Theorem}[section]
\newtheorem{lemma}[theorem]{Lemma}
\newtheorem{proposition}[theorem]{Proposition}
\newtheorem{corollary}[theorem]{Corollary}
\newtheorem{definition}[theorem]{Definition}

\newtheorem{remark}[theorem]{Remark}

\DeclareMathOperator{\FVol}{\mathsf{FVol}}
\DeclareMathOperator{\Area}{\mathsf{Area}}

\begin{document}

\title[A Subgroup Theorem for Homological Filling Functions]{A Subgroup Theorem for Homological Filling Functions}
\author[R.G.~Hanlon]{Richard Gaelan Hanlon}
\author[E.~Mart\'inez-Pedroza]{Eduardo Mart\'inez-Pedroza}
  \address{Memorial University\\ St. John's, Newfoundland, Canada}
  \email{emartinezped@mun.ca}
  \email{gaelanhanlon@gmail.com}
\subjclass[2000]{20F65, 20F67,  20J05,  57M07}
\keywords{Filling Functions, Isoperimetric Functions, Dehn functions, Hyperbolic Groups, Finiteness Properties}

\maketitle

\begin{abstract} We use algebraic techniques to study homological filling functions of groups and their subgroups. If $G$ is a group admitting a finite $(n+1)$--dimensional $K(G,1)$ and $H \leq G$ is of type $F_{n+1}$, then the $n^{th}$--homological filling function of $H$ is bounded above by that of $G$.  
This contrasts with known examples where such inequality does not hold under weaker conditions on the ambient group $G$ or the subgroup $H$. We include applications to hyperbolic groups and homotopical filling functions.
\end{abstract}

\section{Introduction}\label{Intro}

The $n^{th}$ homological and homotopical filling functions of a space are generalized isoperimetric functions describing the minimal volume required to fill an $n$--cycle or $n$--sphere with an $(n+1)$--chain or $(n+1)$--ball.  
These functions have been widely studied in Riemannian Geometry and Geometric Group Theory; see for example~\cite{AWP, Snowflake, WordProc, Ge99, Gromov83, MP15}.  
In this paper, we study the relation between the $n^{th}$ homological filling functions of a finitely presented group and its subgroups. Our main result provides sufficient conditions for the $n^{th}$-filling function of a subgroup to be bounded from above by the $n^{th}$-filling function of the ambient group. The hypotheses of our theorem are in terms of finiteness properties of the ambient group and the subgroup.  Our result contrasts with known examples illustrating that this relation does not hold under weaker conditions~\cite{NonHyp,  YoungHeisenberg, Young13}.  

\subsection{Statement of Main Result} 

A $K(G,1)$ for a group $G$ is a cell complex $X$ with contractible universal cover $\widetilde X$ and fundamental group isomorphic to $G$. If $G$ admits a $K(G,1)$ with finite $n$-skeleton, then $G$ is said to be of type $F_n$. Such finiteness properties are natural (topological) generalizations of being finitely generated (type $F_1$) and finitely presented (type $F_2$). 

If $X$ is a $K(G,1)$ with finite $(n+1)$--skeleton, then the \emph{$n^{\text{th}}$--homological filling function} of $G$ is an optimal function $FV_G^{n+1} \colon \mathbb N \to \mathbb N$ such that $FV_G^{n+1}(k)$ bounds the minimal volume required 
to fill an $n$--cycle $\gamma$ of $\widetilde X$ of volume at most $k$, 
 with an $(n+1)$--chain $\mu$ of $\widetilde X$ having boundary $\partial (\mu) = \gamma$. See Section~\ref{sec:fvg} for precise definitions.  

It can be shown that the growth rate of $FV_G^{n+1}$ is independent of the choice of $X$ up to an equivalence relation $\sim$, hence $FV_G^{n+1}$ is an invariant of the group $G$, see~\cite{Fl98, Young}. The relation $f \sim g$ between functions is defined as $f\preceq g$ and $g\preceq f$, where
$f \preceq g$  means that there is $C > 0$ such that for all $n \in \mathbb N$, $f(n) \leq C g( Cn + C) + Cn + C$.  Our main result is a generalization of a result of Gersten~\cite[Thm C]{Ge92} to higher dimensions.

\begin{theorem}\label{SubgroupDehn} Let $n \geq 1$. Let $G$ be a group admitting a finite $(n+1)$-dimensional $K(G,1)$ and let $H \leq G$  be a subgroup of type $F_{n+1}$. Then $$FV_{H}^{n+1} \preceq FV_{G}^{n+1}.$$
\end{theorem}

Some examples that contrast with Theorem~\ref{SubgroupDehn} are the following.  In~\cite{NonHyp}, Noel Brady constructed a group $G$ admitting a finite $3$--dimensional $K(G,1)$ such that $FV_{G}^2$ is linear, and $G$ contains a subgroup $H \leq G$ of type $F_2$ with $FV_{H}^2$ at least quadratic. Another source of examples are the generalized Heisenberg groups $\mathcal H_{2n+1}$, for which Robert Young computed the homological filling invariants in \cite{YoungHeisenberg, Young13}. For instance, $\mathcal H_5$ admits a finite $5$--dimensional $K(\mathcal H_5, 1)$ and has quadratic $FV_{\mathcal H_5}^2$. On the other hand, $\mathcal H_3$ can be embedded in $\mathcal H_5$, admits a $3$--dimensional $K(\mathcal H_3, 1)$, and has cubic $FV_{\mathcal H_3}^2$.  Likewise, $\mathcal H_5$ has quadratic $FV_{\mathcal H_5}^3$ and can be embedded in $\mathcal H_7$ which has $FV_{\mathcal H_7}^3$ polynomial of degree $3/2$.

Theorem~\ref{SubgroupDehn} also imposes constraints on certain well known constructions. For example, given a finitely generated group $H$ with decidable word problem in nondeterministic polynomial time, 
 Birget, Ol'shanskii, Rips and Sapir produce an embedding of $H$ into a finitely presented group $G$ with polynomial Dehn function~\cite{BORS02}. For this construction, Theorem~\ref{SubgroupDehn} implies that if $H$ has a finite $2$-dimensional $K(H,1)$ and $FV_H^2$ is not bounded by a polynomial function, then $G$ does not admit a finite $2$-dimensional $K(G, 1)$. A particular example of such a  group $H$ is the Baumslag-Solitar group $B(m,n)$ with $|m|\neq|n|$, for which the embedding constraint is known~\cite[Thm A]{Ge92}.

We discuss some applications of Theorem~\ref{SubgroupDehn} to hyperbolic groups and homotopical filling functions below. Recall that a group $G$ is \emph{hyperbolic} if it has a linear Dehn function.
 In~\cite{Ge96}, Gersten proved the following:

\begin{theorem}\label{GerstenSubgroup}\cite[Thm 4.6]{Ge96}  Let $G$ be a hyperbolic group of cohomological dimension $2$. Then every finitely presented subgroup $H \leq G$ is hyperbolic. 
\end{theorem}

Gersten's result does not hold in higher dimensions as Brady has exhibited a hyperbolic group $G$ of cohomological dimension $3$ containing a non--hyperbolic finitely presented subgroup $H \leq G$~\cite{NonHyp}. We can however, obtain a result similar to Theorem~\ref{GerstenSubgroup} by considering homotopical filling functions of higher dimensions. The \emph{$n^{\text{th}}$--homotopical filling function} $\delta_G^n$ of a group $G$ is defined analogously to $FV_G^{n+1}$ but restricts to filling $n$--spheres with $(n+1)$--balls inside the universal cover of $K(G,1)$ with finite $(n+1)$--skeleton. Roughly speaking, $\delta_G^n(k)$ bounds the minimum volume required to fill an $n$--sphere of volume at most $k$, with an $(n+1)$--ball. Precise definitions of ``volume" and $\delta_G^n$ can be found in~\cite{AWP, Snowflake}.

\begin{corollary}\label{GerstenHigherDimension} Let $G$ be a hyperbolic group of geometric dimension $n+1$, where $n \geq 2$. Let $H \leq G$ be of type $F_{n+1}$. Then $\delta_H^n$ is linear.
\end{corollary}

Recall that the \emph{geometric dimension} of a group $G$ is the minimum dimension among $K(G,1)$'s. The Eilenberg--Ganea Theorem~\cite{Brown, EilenbergGanea} states that the cohomological and geometric dimensions of a group $G$ are equal for dimensions greater or equal than $3$. This justifies our use of geometric dimension in the corollary above. In addition to Corollary~\ref{GerstenHigherDimension}, we have the following homotopical version of Theorem~\ref{SubgroupDehn} for sufficiently large $n$.

\begin{corollary}\label{HomotopicalMain} Let $n \geq 3$. Let $G$ be a group admitting a finite $(n+1)$--dimensional $K(G,1)$. Let $H \leq G$ be of type $F_{n+1}$. Then $\delta_H^n \preceq \delta_G^n$.
\end{corollary}

Corollaries~\ref{GerstenHigherDimension} and~\ref{HomotopicalMain} follow from Theorem~\ref{SubgroupDehn} and the following results:

\begin{theorem}\label{HomotopicalHomological}~\cite[pg. 1 and references therein]{ABDY} For $n \geq 3$, the $n^{\text{th}}$--homotopical and homological filling functions $\delta_G^n$ and $FV_G^{n+1}$ are equivalent. For $n =2, \ \delta_G^2 \preceq FV_G^3$.
\end{theorem}

\begin{theorem}\label{ThmLang}\cite{Lang} Let $G$ be a hyperbolic group. Then $FV^{n+1}_G$ is linear for all $n \geq 1$. 
\end{theorem}

\begin{proof}[Proof of Corollary~\ref{GerstenHigherDimension}] 
A theorem of Rips imples that $G$ admits a compact $K(G,1)$, see~\cite{GH90} and then the Eilenberg--Ganea Theorem implies that $G$ admits a compact $(n+1)$--dimensional $K(G,1)$, see~\cite{Brown}.  Theorems~\ref{SubgroupDehn} and~\ref{ThmLang} imply that $FV_H^{n+1}$ is linear. It then follows from  Theorem~\ref{HomotopicalHomological} that $\delta_H^n$ is also linear. 
\end{proof}

\begin{proof}[Proof of Corollary~\ref{HomotopicalMain}] Apply Theorems~\ref{HomotopicalHomological} and~\ref{SubgroupDehn}.
\end{proof}

\begin{remark}Corollary~\ref{GerstenHigherDimension} does not apply to Brady's example $H\leq G$   mentioned above since $H$ is not of type $F_3$.  It is an open question whether or not the subgroups $H$ in Corollary~\ref{GerstenHigherDimension}  are in fact hyperbolic.
\end{remark}

\begin{remark} It is an open question whether or not the statement of Corollary~\ref{HomotopicalMain} holds for $n = 1$ or $2$. In general $\delta_G^1 \nsim FV_G^2$ and $\delta_G^2 \nsim FV_G^3$, examples of such groups are given in~\cite{ABDY, Young}. 
\end{remark}

\subsection{Outline of the Paper}
The rest of the paper is organized into three sections. Section~\ref{sec:fillingnorm} contains the definition of a filling norm on a finitely generated $\ZG$-module and lemmas required for the proof of Theorem~\ref{SubgroupDehn}. Section~\ref{sec:fvg} contains algebraic and topological definitions for $FV_G^{n+1}$. Section~\ref{MainSection} contains the proof of Theorem~\ref{SubgroupDehn}.

\subsection{Acknowledgments}
Thanks to Noel Brady and Mark Sapir for comments on an earlier version of the article.
We especially thank the referee for a list of useful comments and corrections. We acknowledge  funding by the Natural Sciences and Engineering Research Council of Canada, NSERC.

\section{Filling Norms on $\ZG$-Modules}\label{sec:fillingnorm}

In this section we define the notion of a filling-norm on a finitely generated $\ZG$--module.  Most ideas in this section are based on the work of Gersten in~\cite{Ge96}.  The section contains four lemmas on which the proof of the main result of the paper relies on.

\begin{definition}[Norm on Abelian Groups] A \emph{norm} on an Abelian group $A$ is a function $\| \cdot \| \colon A \rightarrow \mathbb R$ satisfying the following conditions:
\begin{itemize}
\item $\| a \| \geq 0$ with equality if and only if $a = 0$, and
\item $ \| a \| + \| a' \| \geq \| a + a' \|$.
\end{itemize}
If, in addition, the norm satisfies
\begin{itemize}
\item $\| n a \| = |n| \cdot \| a \|$, for $n \in \mathbb Z$,
\end{itemize}
then it is called a \emph{regular norm}.
\end{definition}
If $A$ is free Abelian with basis $X$, then $X$ induces a regular $\ell_1$--norm on $A$ given by $\left \| \ \displaystyle \sum_{x \in X} n_x x  \ \right \|_1 = \displaystyle \sum_{x \in X} | n_x |$, where $n_x \in \mathbb Z$.

\begin{definition}[Linearly Equivalent Norms] Two norms $\| \cdot \|$ and $\| \cdot \| '$ on a $\mathbb{Z}$--module $M$ are \emph{linearly equivalent} if there exists a fixed constant $C > 0$ such that \[C^{-1} \| m \| \leq \| m \| ' \leq C \| m \|\] for all $m \in M$.  This is an equivalence relation and the equivalence class of a norm $\|\cdot\|$ is called the \emph{linear equivalence class of $\|\cdot\|$}. 
\end{definition}

\begin{definition}[Based Free $\ZG$--modules and Induced $\ell_1$-norms]\label{def:based}
Suppose $G$ is a group and $F$ is a  free $\ZG$--module with $\ZG-$basis $\{ \alpha_1 , \dots , \alpha_n \}$. Then $\left \{ g \alpha_i  \colon g \in G , 1 \leq i \leq n \right \}$ is a free $\mathbb Z $-basis for $F$ as a (free) $\mathbb Z$-module.  This free $\mathbb Z$--basis  induces a $G$-equivariant $\ell_1$-norm $\| \cdot \|_1$ on $F$. We call a free $\ZG-$module \emph{based} if it is understood to have a fixed basis, and we use this basis for the induced $\ell_1$--norm $\| \cdot \|_1$. 
\end{definition}

\begin{definition}[Filling Norms on $\ZG$-modules] \label{def:filling-norm}
Let $\eta\colon F\to M$ be a surjective homomorphism of $\ZG$-modules and suppose that $F$ is free, finitely generated, and based. The \emph{filling  norm  on $M$ induced by $\eta$ and the free $\ZG$-basis of $F$} is defined as
\[\| m \|_{\eta} =  \min \left \{ \  \| x \|_1 \colon x \in F , \ \eta(x) = m \right \}.\] 
Observe that this norm is $G$-equivariant.
\end{definition} 

\begin{remark}
Gersten observed that filling norms are not in general regular norms. He illustrated this fact with the following example~\cite{Ge96-2}. Let $X$ be the universal cover of the standard complex of the group presentation $\langle x | x^2, x^{2k} \rangle$, where $k\geq 2$. The filling norm on  the integral cellular $1$-cycles $Z_1(X)$ induced by $C_2(X)\overset\partial\to Z_1(X)$ is not regular since $\|2x\|_\partial=\|2kx\|_\partial=1$.
\end{remark}

\begin{remark}[Induced $\ell_1$-norms are Filling Norms]\label{rem:l1}
If $F$ is a finitely generated based free $\ZG$-module, then the  $\ell_1$-norm induced by a free $\ZG$-basis is a filling norm. 
\end{remark}

The following lemma is reminiscent of the fact that linear operators on finite dimensional normed spaces are bounded. 

\begin{lemma}[$\ZG$-Morphisms between Free Modules are Bounded]\label{EquivFree}\cite[Proof of Proposition 4.4]{Ge96} Let $\varphi : F \rightarrow F'$ be a homomorphism between finitely generated, free, based $\ZG-$modules. Let $\| \cdot \|_1$ and $\| \cdot \|_1'$ denote the induced $\ell_1$--norms of $F$ and $F'$. Then there exists a constant $C>0$ such that for all $x \in F$ 
\[ \| \varphi(x) \|_1 ' \leq C\cdot \| x \|_1. \]
\end{lemma}

\begin{proof} 
Let $A = \{ \alpha_1, \dots ,\alpha_n \}$ 
be the  $\ZG-$basis of $F$ inducing the norm $\| \cdot \|_1$. 
Then $\varphi$ is given by a finite $n \times m$ matrix whose entries are elements of $\ZG$. 
Observe that for any $g \in G, \ x \in F$, we have 
$\| x \|_1 = \| gx \|_1$. 
Define $C= \displaystyle \max_{1\leq i \leq n} \left \{ \|\varphi (\alpha_i) \| \right \}$ and let $x \in F$ be arbitrary. Then
 \begin{align} \nonumber 
 \| \varphi(x) \|_1' & = \| \varphi( \lambda_1 \alpha_1 + \dots + \lambda_n \alpha_n) \|_1', \text{ where } \lambda_i \in \mathbb Z G \\ \nonumber 
& \leq \left \| \ \left ( \displaystyle\sum_{g \in G} \lambda_{1,g} g \right ) \varphi(\alpha_1) \right \|_1' + \dots + \left \| \left ( \displaystyle\sum_{g \in G} \lambda_{n,g} g \right ) \varphi(\alpha_n) \ \right \|_1', \text{ where } \lambda_j = \displaystyle\sum_{g \in G} \lambda_{j,g} g  \text{ and }  \lambda_{j,g} \in \mathbb Z  \\ \nonumber 
& \leq  \left ( \displaystyle\sum_{g \in G} |\lambda_{1,g}| \right ) \| \varphi(\alpha_1)\|_1' + \dots + \left ( \displaystyle\sum_{g \in G} |\lambda_{n,g}| \right ) \| \varphi(\alpha_n)\|_1' && \\ \nonumber
& \leq C\left ( \displaystyle\sum_{i =1}^{m} \left ( \displaystyle\sum_{g \in G} |\lambda_{i,g}| \right ) \right ) = C \| x \|_1.  \qedhere
 \end{align}
\end{proof}

\begin{lemma}[$\ZG$-Morphisms with Projective Domain are Bounded]\label{lem:ProjBounded}
Let $\varphi : P \rightarrow Q$ be a homomorphism between finitely generated $\ZG-$modules. Let $\| \cdot \|_P$ and $\| \cdot \|_Q$ denote filling norms on $P$ and $Q$ respectively. If $P$ is projective then there exists a constant $C>0$ such that for all $p \in P$ 
\[ \| \varphi(p) \|_Q \leq C \cdot\| p \|_P .\]
\end{lemma}

\begin{proof}
Consider the commutative diagram 
\[\xymatrix{ 
A \ar[r]^{\tilde \varphi} \ar[d]_\rho & B \ar[d] \\
P \ar[r]_\varphi  \ar[ru]^\psi  & Q
}\]
constructed as follows.  Let $A$ and $B$ be finitely generated and based free $\ZG$-modules,  and let  $A\to P$ and $B\to Q$ be surjective morphisms inducing the filling norms $\| \cdot \|_P$ and $\| \cdot \|_Q$.  Since $P$ is projective and $B\to Q$ is surjective, there is a lifting $\psi \colon P \to B$ of $\varphi$; then let $\tilde \varphi$ be the composition $A \overset\rho\to P \overset\psi\to B$. Let $C$ be the constant provided by Lemma~\ref{EquivFree} for $\tilde \varphi$. Let $p\in P$ and let $a \in A$ that maps to $p$. It follows that 
\begin{equation}\nonumber
 \|\varphi (p)\|_Q  \leq \| \psi (p) \|_1 = \| \tilde \varphi (a) \|_1  \leq C  \| a \|_1.
\end{equation}
Since the above inequality holds for any $a \in A$ with $\rho (a)=p$, it follows that 
\begin{equation}\nonumber
\begin{split}
  \|\varphi (p)\|_Q & \leq C \cdot \min_{\rho (a)= p} \left\{\|a\|_1  \right \} \\ 
&= C\cdot \|p\|_P. \qedhere 
\end{split}
\end{equation}
\end{proof}

\begin{lemma}[Equivalence of Filling Norms for $\ZG$-modules]\label{EquivStd}\cite[Lemma 4.1]{Ge96} Any two filling norms on a finitely generated $\ZG$--module $M$ are linearly equivalent. 
\end{lemma}

\begin{proof} Consider a pair of surjective homomorphisms of $\ZG-$modules $\eta : F \rightarrow M$ and $\eta' : F' \rightarrow M$ such that $F$ and $F'$ are finitely generated, free,  based modules inducing the filling norms $\|\cdot\|_\eta$ and $\|\cdot\|_{\eta'}$ on $M$. Since $\eta'$ is surjective, the universal property of $F$ provides a homomorphism $\varphi$ such that $\eta = \eta' \circ \varphi$. Let $m \in M$ be arbitrary and take $x \in F$ such that $\eta(x) = m$. Since $\eta' \circ \varphi (x) = m$, by Lemma \ref{EquivFree} there exists $C > 0$ such that 
$$\| m \|_{\eta'} = \displaystyle \min_{\eta'(x') = m} \| x' \|_1' \leq \|\varphi (x) \|_1' \leq C\cdot \| x \|_1.$$ As this inequality holds for all $x \in F$ satisfying $\eta(x) = m$, we have $$\|m\|_{\eta'} \leq C\cdot \min_{\eta(x)=m}\left\{\| x \|_1\right\} = C\cdot \| m \|_{\eta}.$$ The other inequality proceeds in a similar manner.
\end{proof}

\begin{lemma}[Retraction Lemma]\label{RetractionLemma}\cite[Prop. 4.4]{Ge96} Let $0 \rightarrow M \overset{\iota} \rightarrow N \rightarrow P \rightarrow 0$ be a short exact sequence of $\ZG-$modules where
\begin{itemize}
\item[1)] $M$ is finitely generated and equipped with a filling-norm $\|\cdot\|_M$. 
\item[2)] $N$ is free, based, and equipped with the induced $\ell_1$-norm $\| \cdot \|_1$.
\item[3)] $P$ is projective.
\end{itemize}
Then there exists a retraction $\rho : N \rightarrow M$ for the inclusion $\iota : M \rightarrow N$ and a fixed  constant $C > 0$ such that $\| \rho(x) \|_M \leq C \| x \|_1$ for all $x \in N$.
\end{lemma}

\begin{proof} Since $P$ is projective there is a retraction $\rho'$ for $\iota$. Since $M$ is finitely generated, 
 $N$ is isomorphic to a product $I\oplus Q$ of free modules where $I$ is finitely generated and contains the image of $M$. Define $\rho : N \rightarrow M$ by $\rho |_{I} = \rho' |_{I}$ and $\rho |_{Q} = 0$. Then $\rho$ is a retraction for $\iota$ with support contained in $I$.

Each $x \in N$ has a unique decomposition $x = y + q$ where $y \in I, \ q \in Q$ such that $\rho(x) = \rho(y)$ and $\| y \|_1 \leq \| x \|_1$. Apply Lemma~\ref{lem:ProjBounded} to the restriction $\rho \colon I \rightarrow M$ to obtain $C > 0$ such that $$ \| \rho(x) \|_M = \| \rho(y) \|_M \leq C \| y \|_1 \leq C \| x \|_1. $$\qedhere
\end{proof}

\section{Homological Filling Functions of Groups}\label{sec:fvg}

In this section, given a group $G$  of type $FP_{n+1}$, where $n \geq 1$, we define the group invariant $FV_G^{n+1}$.  In the first part of the section we provide an algebraic definition of $FV_G^{n+1}$ and prove that it is well defined. This algebraic approach, while naturally inspired by the topological approach, provides a convenient algebraic framework suitable for some of the arguments in this paper. This algebraic approach has been also explored in~\cite{JOR13}.  In the second part, we recall the topological approach to $FV_G^{n+	1}$ and show that the topological and algebraic approaches are equivalent for finitely presented groups of type $FP_{n+1}$. The final subsection discusses why $FV_G^{n+1}(k)$ is a finite number.

\subsection{Algebraic Definition of $FV_G^{n+1}$}\label{FVAlgebraic}

\begin{definition}[Linearly Equivalent Functions]\label{def:lin-equiv} Let $f$ and $g$ be functions from $\mathbb N$ to $\mathbb N$. Define $f \preceq g$ if there exists $C > 0$ such that for all $n \in \mathbb N$ \[f(n) \leq C g( Cn + C) + Cn + C.\] The functions $f$ and $g$ are \emph{linearly equivalent},  $f \sim g$,  if both $f \preceq g$ and $g \preceq f$ hold.    This is an equivalence relation and the equivalence class containing a function $f$ is called the \emph{linear equivalence class of $f$}. 
\end{definition}

\begin{definition}[$FP_{n}$  group]\cite{Brown}
A group $G$ is of \emph{type $FP_{n}$} if there is a resolution of $\ZG$--modules 
\[  P_{n} \overset{\partial_{n}} \longrightarrow P_{n-1} \overset{\partial_{n-1}} \longrightarrow \dots \overset{\partial_{2}} \longrightarrow P_1 \overset{\partial_{1}} \longrightarrow P_0 \longrightarrow \mathbb Z \rightarrow 0,\] 
such that for each $i\in \{0, 1 \ldots , n\}$ the module $P_i$ is a finitely generated  projective $\ZG$--module. In this case, such a resolution is called an \emph{$FP_{n}$--resolution}.
\end{definition}

\begin{definition}[Algebraic definition of $FV_G^{n+1}$] \label{def:algFVG} Let $G$ be a group of type $FP_{n+1}$. The algebraic \emph{$n^\text{th}$--filling function} is the (linear equivalence class of the) function \[FV_{G}^{n+1}\colon \N \to \N\] defined as follows. 
Let 
\[ P_{n+1} \overset{\partial_{n+1}} \longrightarrow P_n \overset{\partial_{n}} \longrightarrow \dots \overset{\partial_{2}} \longrightarrow P_1 \overset{\partial_{1}} \longrightarrow P_0 \longrightarrow \mathbb Z \rightarrow 0,\] 
 be a resolution of $\ZG$--modules for $\mathbb Z$ of type $FP_{n+1}$. Choose filling norms for  $P_n$ and $P_{n+1}$,  denoted by $\|\cdot \|_{P_n}$ and $\|\cdot \|_{P_{n+1}}$ respectively. 
 Then
\[FV_{G}^{n+1}(k) = \max \left \{ \ \| \gamma \|_{\partial_{n+1}} \colon \gamma \in \ker(\partial_n), \ \| \gamma \|_{P_n} \leq k \ \right \},\]
where 
\[ \| \gamma \|_{\partial_{n+1}} = \min \left\{  \| \mu\|_{P_{n+1}} \colon \mu \in P_{n+1},\  \partial_{n+1} (\mu) =\gamma \right\} .\]
\end{definition}

\begin{remark}[Finiteness of $FV_G^{n+1}$]\label{rem:FVGfinite} It is not immediately clear that the maximum in Definition~\ref{def:algFVG} is a finite number.  In Section~\ref{sec:FVGfinite} we recall some results from the literature which, under the assumption $G$ is finitely presented, imply that $FV_G^{n+1}$ is a finite valued function for $n=1$ and $n \geq 3$. The authors are not aware of a proof for the case $n=2$. 

For $n=2$, all results in this paper regarding $FV_G^3$ hold under the following natural modifications. First, work with the standard extensions of addition, multiplication, and order, of the positive integers $\mathbb N$ to $\mathbb N \cup \{ \infty \}$. Definition~\ref{def:lin-equiv} is extended to functions $\mathbb N \rightarrow \mathbb N \cup \{ \infty \}$, but we emphasize that the constant $C$ remains a finite positive integer. In Definition~\ref{def:algFVG} the function $FV^3_G$ is defined as an $\mathbb N \rightarrow \mathbb N \cup \{ \infty \}$ function. We observe that no argument in this paper relies on  $FV_G^{n+1}(k)$ being finite. 
\end{remark}

\begin{theorem}[$FV_G^{n+1}$ is a Well-defined Group Invariant]\label{AlgebraicDef}  Let $G$ be a group of type $FP_{n+1}$. Then the algebraic $n^\text{th}$--filling function $FV_G^{n+1}$ of $G$ is well defined up to linear equivalence.
\end{theorem}
\begin{proof} Let $(F_\ast, \partial_\ast)$ and $(P_\ast, \delta_\ast)$ be a pair of resolutions of $\ZG$-modules of type $FP_{n+1}$ with choices of filling-norms for  their $n^\text{th}$ and  $(n+1)^\text{th}$ modules denoted by $\|\cdot \|_{F_n}$ and $\|\cdot \|_{F_{n+1}}$, and $\|\cdot \|_{P_n}$ and $\|\cdot \|_{P_{n+1}}$ respectively.  Let $FV^{n+1}_{F_\ast}$  and $FV^{n+1}_{P_{\ast}}$ be the induced functions according to Definition~\ref{def:algFVG}. By symmetry, it is enough to show that $FV^{n+1}_{F_\ast} \preceq FV^{n+1}_{P_{\ast}}$.

It is well known that any two projective resolutions of a $\ZG$-module are chain homotopy equivalent, see for example~\cite[pg.24, Thm 7.5]{Brown}\label{ChainHomotopy}, and hence the resolutions $F_\ast$ and $P_\ast$ are chain homotopy equivalent.  Therefore there exists chain maps $f_i : F_i \rightarrow P_i, \ g_i: P_i \rightarrow F_i$, and a map $h_i : F_i \rightarrow F_{i+1}$ such that \[ \partial_{i+1} \circ h_i + h_{i-1} \circ \partial_i = g_i \circ f_i - id.\]

Let $C$ denote the maximum of the constants for the maps $g_{n+1}$,  $h_n$, and $f_n$ and the chosen filling-norms  provided by Lemma~\ref{lem:ProjBounded}. We claim that for every $k\in \N$, 
\[  FV^{n+1}_{F_\ast}(k) \leq C\cdot FV^{n+1}_{P_\ast} ( Ck + C ) + Ck + C. \]
Fix $k$.  Let $\alpha \in \ker( \partial_n)$ be such that 
$\| \alpha \|_{F_n} \leq k$. 
Choose $\beta \in P_{n+1}$ such that $\delta_{n+1}(\beta) = f_n (\alpha )$ and $\|f_n(\alpha) \|_{\delta_{n+1}} = \| \beta \|_{P_{n+1}}$. 
By commutativity of the chain maps and the chain homotopy equivalence,
\begin{equation}\nonumber
\begin{split}
\partial_{n+1} \circ h_n ( \alpha ) + h_{n-1} \circ \partial_n (\alpha ) & = g_n \circ f_n ( \alpha ) - \alpha \\
& = g_n \circ \delta_{n+1}(\beta )  - \alpha \\
& = \partial_{n+1} \circ g_{n+1} (\beta ) - \alpha. 
\end{split}
\end{equation}
Since $\alpha \in \ker( \partial_n)$, we have that $h_{n-1} \circ \partial (\alpha) = 0$. Rearranging the above equation, we obtain 
\[ \alpha = \partial_{n+1} \circ g_{n+1} (\beta ) - \partial_{n+1} \circ h_n ( \alpha ) = \partial_{n+1} \left (g_{n+1} (\beta )- h_n ( \alpha ) \right).\] 
Hence $g_{n+1}(\beta) - h_n(\alpha)$ has boundary $\alpha$. Observe that
\begin{align} \nonumber
\| \alpha \|_{\partial_{n+1}} & \leq \left \| g_{n+1}(\beta) - h_n(\alpha) \right \|_{F_{n+1}} &&  \text{since $\partial_{n+1}( g_{n+1}(\beta) - h_n(\alpha)) = \alpha$} \\ \nonumber
& \leq\left \| g_{n+1}(\beta)  \right \|_{F_{n+1}} + \left \| h_n(\alpha ) \right \|_{F_{n+1}} &&  \text{by the triangle inequality} \\ \nonumber
& \leq C  \cdot \| \beta \|_{P_{n+1}} + C \cdot  \| \alpha \|_{F_{n}} && \text{by Lemma~\ref{lem:ProjBounded}}\\  \nonumber
& =  C \cdot \left \| f_n( \alpha ) \right \|_{\delta_{n+1}} + C \cdot \| \alpha \|_{F_{n}} && \text{by definition of $\beta$} \\  \nonumber
& \leq C \cdot FV^{n+1}_{P_\ast} ( \| f_n(\alpha ) \|_{P_n} ) + C \| \alpha \|_{F_n} && \text{by definition of $FV^{n+1}_{P_\ast}$} \\  \nonumber
& \leq C \cdot FV^{n+1}_{P_\ast} \left ( C \| \alpha \|_{F_n} \right) + C  \| \alpha \|_{F_n} && \text{by Lemma~\ref{lem:ProjBounded}}  \\  \nonumber
& \leq C\cdot FV^{n+1}_{P_\ast} ( Ck + C ) + Ck + C && \text{since $\| \alpha \|_{F_n} \leq k$.}  
\end{align}  
Since $\alpha$ was arbitrary, $FV_{F_\ast}^{n+1}(k) \leq C\cdot FV^{n+1}_{P_\ast} ( Ck + C ) + Ck + C$ for all $k \in \mathbb N$. This shows that $FV^{n+1}_{F_\ast} \preceq FV^{n+1}_{F'_{\ast}}$ completing the proof.
\end{proof}

\subsection{Topological Definition of $FV_G^{n+1}$}

For a cell complex $X$, the cellular chain group $C_i(X)$ is a free Abelian group with basis the collection of all $i$-cells of $X$. This natural basis induces an $\ell_1$-norm on $C_i(X)$ that we denote by $\|\cdot\|_1$. Recall that a complex $X$ is $n$-connected if its first $n$-homotopy groups are trivial.

\begin{definition}[$F_n$ group]
A group $G$ is of type $F_n$ if there is a $K(G, 1)$-complex with a finite $n$-skeleton, i.e., with only finitely many cells in dimensions $\leq n$.
\end{definition}

\begin{definition}[Topological Definition of $FV_G^{n+1}$]\cite{Fl98, Young}\label{def:FVG} Let $G$ be a group acting properly, cocompactly, by cellular automorphisms on an $n$--connected 
 cell complex $X$.   The topological \emph{$n^\text{th}$--filling function} of $G$ is the (linear equivalence class of the) function $FV_{G}^{n+1}\colon\N \to \N$ defined as 
\[FV^{n+1}_G (k) = \max \left \{ \ \| \gamma   \|_{\partial}  \colon \gamma \in Z_n(X), \ \| \gamma \|_1 \leq k \ \right \},\]
where 
\[ \| \gamma  \|_\partial  = \min \left \{ \ \| \mu \|_1 \colon \mu \in C_{n+1}(X), \ \partial ( \mu ) = \gamma \ \right \}.\] 
\end{definition}

J.~Fletcher and R.~Young have independently provided geometric proofs that the topological $n^{\text{th}}$--filling function $FV_G^{n+1}$ is well defined as an invariant of the group, see~\cite[Theorem 2.1]{Fl98} and~\cite[Lemma 1]{Young} respectively. In the work of Fletcher, the topological definition of $FV_G^{n+1}$ requires $X$ to be the universal cover of a $K(G,1)$, while Young's proof is in the more general context introduced above.

\begin{theorem}\label{ThmYoung}\cite{Young}
Let $G$ be a group admitting a proper and cocompact action by cellular automorphisms on an $n$--connected cell complex. Then the topological  $n^\text{th}$--filling invariant $FV_G^{n+1}$ of $G$ is well defined up to linear equivalence.
\end{theorem}

 Even in the topological definition, it is not trivial that $FV_G^{n+1}$ is a finite valued function and Remark~\ref{rem:FVGfinite} also applies in this case. For the rest of the section, we show that the topological and algebraic approaches to $FV_G^{n+1}$ are equivalent for finitely presented groups of type $FP_{n+1}$.

\begin{proposition}\label{prop:Fn1}
Let $n \geq 1$ and let $G$ be a group of type $F_{n+1}$. Then $G$ is of type $FP_{n+1}$ and the algebraic and topological $n^\text{th}$--filling functions of $G$ are linearly equivalent.
\end{proposition}

\begin{proof}
Let $X$ be a $K(G,1)$ with finite $(n+1)$-skeleton. The $G$-action on the universal cover $\widetilde X$ of $X$ induces the structure of a $\ZG$-module to the group of cellular chains  $C_i(\widetilde X)$ and each boundary map $\partial_i$ is a morphism of $\ZG$-modules.  Since the $G$-action on $\widetilde X$ is cellular and free, each $C_i(\widetilde X)$ is a free $\ZG$-module with basis any collection of representatives of the $G$-orbits of $i$-cells.  Since the action is cocompact on the $(n+1)$--skeleton, each $C_i(\widetilde X)$ is a finitely generated free $\ZG$-module for $i \in \{0, 1, \dots , n+1 \}$. Since $\widetilde X$ is a contractible space, all its homology groups are trivial and therefore we have a resolution of $\ZG$-modules
\[ \cdots \longrightarrow  C_{n+1}(X) \overset{\partial_{n+1}} \longrightarrow C_n(X) \overset{\partial_{n}} \longrightarrow \dots \overset{\partial_{2}} \longrightarrow C_1(X) \overset{\partial_{1}} \longrightarrow C_0(X) \longrightarrow \mathbb Z \rightarrow 0,\] 
of type $FP_{n+1}$.  Under our assumptions, the induced topological  $n^\text{th}$--filling function of $G$ is a particular instance of an algebraic $n^\text{th}$--filling function of $G$. The conclusion then follows from Theorems~\ref{AlgebraicDef} and~\ref{ThmYoung}.
\end{proof}

\begin{proposition}\cite[pg 205, proof of Thm. 7.1]{Brown}\label{prop:Fn11}  Let $G$ be finitely presented and of type $FP_n$ where $n \geq 2$. Then $G$ is of type $F_n$.
\end{proposition}

Propositions~\ref{prop:Fn1} and~\ref{prop:Fn11}  imply the following statement.

\begin{corollary}\label{EquivGeomAlg2} Let $G$ be a finitely presented group of type $FP_{n+1}$. Then the topological and algebraic definitions of $FV_G^{n+1}$ are equivalent.
\end{corollary}

\subsection{Finiteness of $FV_G^{n+1}(k)$}\label{sec:FVGfinite} Let $G$ be a finitely presented group of type $FP_{n+1}$, or equivalently assume that $G$ is of type $F_{n+1}$; see Proposition~\ref{prop:Fn11}. We will sketch why $FV_G^{n+1}$ is a finite valued function for $n=1$ and $n \geq 3$.

\subsubsection{Case $n=1$} Finiteness of $FV_G^2$ follows from that of the Dehn function $\delta_G$. We summarize the argument from Gersten's article~\cite[Prop $2.4$]{Ge99}. Let $X$ be a $K(G,1)$ with finite $2$--skeleton and let $z \in Z_1 \left ( \widetilde X \right )$ be a $1$--cycle with $\| \gamma \|_1 \leq k$. Then $z = z_1 + \dots z_m$ for some $m \leq k$ where each $z_i$ is the $1$--cycle induced by a simple edge circuit $\gamma_i$ in $\widetilde X$ and $\displaystyle \sum_{i=1}^{m} \ell(\gamma_i) = \| z \|_1$. Then $$\| z \|_{\partial_2} \leq \displaystyle \sum_{i=1}^{m} \Area(\gamma_i) \leq \displaystyle \sum_{i=1}^{m} \delta_G \left ( \ell(\gamma_i) \right ) \leq k \cdot \delta_G(k) < \infty.$$ 

\subsubsection{Case $n \geq 3$} A group $G$ of type $F_{n+1}$ has a well defined invariant called the $n^{th}$--homotopical filling function $\delta_G^n \colon \mathbb N \rightarrow \mathbb N$. There are multiple approaches to define $\delta_G^n$, we sketch the approach found in~\cite{ABDY, Snowflake}. Roughly speaking, if $X$ is a $K(G,1)$ with finite $(n+1)$--skeleton, then $\delta_G^n(k)$ measures the number of $(n+1)$--balls required to fill a sphere $S^n \to \widetilde X$ comprised of at most $k$ $n$--balls. Here the maps $f \colon S^n \to \widetilde X$ and fillings $\tilde f \colon D^{n+1} \to \widetilde X$ are required to be in a particular class of maps called \emph{admissible maps}. This allows one to define the volumes, $vol(f)$ and $vol(\tilde f)$, as the number of $n$-balls and $(n+1)$--balls of $S^n$ and $D^{n+1}$ respectively, mapping homeomorphically to open cells  of $\widetilde X$. The \emph{filling volume} of $f$ is given by $$\FVol(f) = \sup \{ \ vol( \tilde f ) \ | \ \tilde f \colon D^{n+1} \rightarrow \widetilde X, \ \tilde f|_{\partial D^{n+1}} = f \ \}$$ and $\delta_G^n$ by $$\delta_G^n(k) = \max \{ \ \FVol(f) \ | f \colon S^n \rightarrow \widetilde X, \ vol(f) \leq k \ \}.$$ Alonso et al. use higher homotopy groups as $\pi_1(X)$--modules to provide a more algebraic approach to $\delta_G^n$, in particular they show that $\delta_G^n$ is a finite valued function~\cite[Corollary $1$]{AWP}. It is observed in \cite[Remark $2.4(2)$]{Snowflake} that Alonso's approach and the approach described above are equivalent. 

The finiteness of $FV_G^{n+1}$ then follows from the inequality $$FV_G^{n+1} \preceq \delta_G^n$$ which holds for all $n \geq 3$. We outline the argument for this inequality described in the introduction of \cite{ABDY}. Let $X$ be a $K(G,1)$ with finite $(n+1)$--skeleton and let $\gamma \in Z_n \left ( \widetilde X \right )$ with $\| \gamma \|_1 \leq k$. Using the Hurewicz Theorem, one can show (see~\cite{Gromov83, White}) that $\gamma$ is the image of the fundamental class of an $n$--sphere for a map $f \colon S_n \rightarrow \widetilde X$ such that $vol(f) = \| \gamma \|_1$. If $\tilde f \colon D^{n+1} \rightarrow \widetilde X$ is an extension of $f$ to the $(n+1)$--ball $D^{n+1}$, then the image of the fundamental class of $D^{n+1}$ is an $(n+1)$--chain $\mu$ with $\partial(\mu) = \gamma$ and $vol \left ( \tilde f \right ) \geq \| \mu \|_1$. Therefore the filling volume $$\FVol(f) = \sup \{ \ vol( \tilde f ) \ | \ \tilde f \colon D^{n+1} \rightarrow \widetilde X, \ \tilde f|_{\partial D^{n+1}} = f \ \}$$ is greater than or equal to the filling norm $\| \gamma \|_{\partial_{n+1}}$. It follows that $FV_G^{n+1}(k) \leq \delta_G^n(k)$

\section{Main Result}\label{MainSection} As we will be working with cell complexes, all relevant computations in this section are understood to occur within cellular chain complexes. 

\begin{definition}[Stably free] A $\ZG$-module $P$ is   \emph{stably  free} if there exists finitely generated free $\ZG$ module $F$ such that $P \oplus F$ is free.
\end{definition}

\begin{lemma}[The Eilenberg Trick]\cite[pg.207]{Brown}\label{StdTrick} Let $G = \pi_1(X,x_0)$, where $X$ is a cell complex.  Then $X$ is a subcomplex of a complex $Y$ such that the inclusion $X\hookrightarrow Y$ is a homotopy equivalence, and 
the cellular $n$-cycles of the universal covers $\widetilde Y$ and $\widetilde X$ satisfy
$Z_n \left ( \widetilde Y \right ) \simeq Z_n \left ( \widetilde X \right ) \oplus \ZG$
as $\ZG$-modules.
\end{lemma}

\begin{proof} Let $x_0$ be a $0$-cell of $X$, and glue an $n$-cell $D^n$ to $(X,x_0)$ by mapping its boundary to $x_0$. The resulting space is the wedge sum of $X$ and an $n$-sphere $S^n$. To obtain $Y$,  attach an $(n+1)$-cell $D^{n+1}$ by the attaching map that identifies $\partial D^{n+1}$ with the $n$-sphere $S^n$. Then  $Z_n \left ( \widetilde Y \right ) \simeq Z_n \left ( \widetilde X \right ) \oplus \ZG$ where the $\ZG$ factor is generated by a lifting of the $n$-cell $D^n$ to $\widetilde Y$. It is clear that $X \hookrightarrow Y$ is a homotopy equivalence.
\end{proof}

\begin{lemma}[Schanuel's Lemma]\cite[pg.193, Lemma 4.4]{Brown} Let $$0 \rightarrow P_n \rightarrow P_{n-1} \rightarrow \dots \rightarrow P_0 \rightarrow M \rightarrow 0$$ and $$0 \rightarrow P_n' \rightarrow P_{n-1}' \rightarrow \dots \rightarrow P_0' \rightarrow M \rightarrow 0$$ be exact sequences of $R$--modules with $P_i$ and $P_i'$ projective for $i \leq n-1$. Then $$P_0 \oplus P_1' \oplus P_2 \oplus P_3' \oplus \dots \ \large \simeq \ P_0' \oplus P_1 \oplus P_2' \oplus P_3 \oplus \dots $$
\end{lemma}

We are now ready to prove our main result which is a generalization of~\cite[Thm C]{Ge92}. The proof is based on Gersten's proof of \cite[Thm 4.6]{Ge96} and is adjusted for higher dimensions:

\begin{theorem}\label{Main} Let $G$ be a group admitting a finite $(n+1)$-dimensional $K(G,1)$ and let $H \leq G$  be a subgroup of type $F_{n+1}$. Then
$FV_H^{n+1} \preceq FV_G^{n+1}.$
\end{theorem}

\begin{proof} Let $W$ be a finite $(n+1)$-dimensional $K(G,1)$.  Let $X$ be the $(n+1)$-skeleton of a $K(H,1)$. Since $H$ is of type $F_{n+1}$, we may assume that $X$ is a finite cell complex.  Then, after subdivisions, there exists a cellular map $f : X \rightarrow W$ inducing the inclusion $H \hookrightarrow G$ at the level of fundamental groups. Let $M_f$ be the mapping cylinder of $f$ and consider the  exact sequences of $\ZG$-modules 
\begin{equation}
  0 \to  Z_{n} \left ( \widetilde M_f \right ) \to   C_{n} \left ( \widetilde M_f \right ) \to \cdots \to  C_0 \left ( \widetilde M_f \right ) \to \mathbb Z \to 0
\end{equation}
and
\begin{equation}
  0 \to C_{n+1} \left ( \widetilde W \right ) \to  C_n \left ( \widetilde W \right )  \to  \cdots \to C_0 \left ( \widetilde W \right ) \to  \mathbb Z \to 0,
\end{equation}
where $\widetilde W$ and $\widetilde M_f$ denote the universal covers of $W$ and $M_f$ respectively. 

Applying Schanuel's lemma to the above sequences shows that the $\ZG-$module $Z_n \left ( \widetilde M_f \right )$ is finitely generated and stably   free. Let $Y$ be the space obtained by attaching a finite number of $(n+1)$--balls to the base point of $M_f$ as in Lemma \ref{StdTrick} such that $Z_n \left ( \widetilde Y \right )$ is finitely generated and free as a $\ZG$-module.  

From here on, we are only concerned with the inclusion map $X \to Y$ realizing the inclusion $H \to G$ at the level of fundamental groups with the property that $Z_n \left ( \widetilde Y \right )$ is finitely generated and free as a $\ZG$-module.   Since the inclusion $X \to Y$ is injective at the level of fundamental groups, any lifting $\widetilde X \to \widetilde Y$ is an embedding. Moreover, we can choose the lifting to be  equivariant with respect to the inclusion $H\to G$. Without loss of generality,  assume that $\widetilde X$ is an $H$-equivariant subcomplex of $\widetilde Y$. 

Since the ring $\ZG$ is free as a $\mathbb ZH$--module, it follows that $C_i \left ( \widetilde Y \right )$ is a free $\mathbb ZH$-module. Since $\widetilde X$ is an $H$-equivariant subcomplex of $\widetilde Y$,  the $\mathbb Z H-$module $C_i \left (\widetilde{X} \right )$ is a free factor of $C_i \left ( \widetilde Y \right )$.  Hence the quotient $C_i \left ( \widetilde{Y}^{(n)}  , \widetilde{X}^{(n)} \right ) =  C_i \left ( \widetilde Y \right ) / C_i \left ( \widetilde X \right )$ is a free $\Z H-$module. 

Restricting our attention to $n$-skeleta, the following short exact sequence of chain complexes of $\mathbb Z H$-modules arises 
\begin{equation} 0 \rightarrow C_{\ast} \left ( \widetilde{X}^{(n)} \right ) \rightarrow C_{\ast} \left ( \widetilde{Y}^{(n)} \right ) \rightarrow C_{\ast} \left ( \widetilde{Y}^{(n)} \ , \ \widetilde{X}^{(n)} \right )  \rightarrow 0.\end{equation}  
Consider the induced long exact homology sequence 
\begin{equation}\label{eq:HXY} 0 \rightarrow \widetilde H_n \left ( \widetilde{X}^{(n)} \right ) \rightarrow  \widetilde H_n \left ( \widetilde{Y}^{(n)} \right ) \rightarrow   \widetilde H_n\left ( \widetilde{Y}^{(n)}  , \widetilde{X}^{(n)} \right ) \rightarrow  \widetilde H_{n-1} \left ( \widetilde{X}^{(n)} \right ) \rightarrow  \cdots .\end{equation}
Since $X$ is the $(n+1)$-skeleton of an $K(H,1)$, the homology group $ \widetilde H_{n-1}(\widetilde{X}^{(n)})$ is trivial.  Now the exact sequence~\eqref{eq:HXY} can be truncated, obtaining the short exact sequence
\begin{equation}\label{SES}
0 \rightarrow Z_n \left ( \widetilde{X} \right ) \overset\iota\rightarrow Z_n \left ( \widetilde{Y} \right ) \rightarrow Z_n \left ( \widetilde{Y} \ , \ \widetilde{X} \right ) \rightarrow 0,
\end{equation}
where $\iota$ is induced by the inclusion $\widetilde X \subseteq \widetilde Y$. We claim that the short exact sequence~\eqref{SES} satisfies the three hypothesis of Lemma~\ref{RetractionLemma}.

First, since $X$ is a finite cell complex, $C_{n+1} \left ( \widetilde{X} \right )$ is finitely generated as a $\mathbb ZH$-module. Therefore $Z_n \left ( \widetilde{X} \right )$ is also finitely generated as a $\mathbb ZH$-module.

Second, the construction of $Y$ guarantees that $Z_n \left ( \widetilde Y \right )$ is a free $\ZG$--module, hence $Z_n \left ( \widetilde Y \right )$ is a free $\mathbb ZH$-module.

Third, we need to verify that  $Z_{n} \left ( \widetilde{Y},\widetilde{X} \right )$ is a projective $\Z H$-module; in fact we show that it is stably free. Indeed,  since $X^{(n)}$ and $Y^{(n)}$ are the $(n+1)$-skeletons of a $K(H,1)$ and a $K(G,1)$ respectively, the reduced homology groups  $ \widetilde H_{k}\left(\widetilde{X}^{(n)}\right)$  and   $ \widetilde H_{k}\left(\widetilde{Y}^{(n)}\right)$  are trivial for $1\leq k<n$. Then, considering the exact sequence~\eqref{eq:HXY}, we have that
\begin{equation}\label{eq:CYX}
  0 \to Z_{n} \left ( \widetilde{Y}^{(n)},\widetilde{X}^{(n)} \right ) \to C_n \left ( \widetilde{Y}^{(n)} , \widetilde{X}^{(n)} \right )   \to \cdots \to C_0 \left ( \widetilde{Y}^{(n)} , \widetilde{X}^{(n)} \right ) \to    0  
\end{equation}
is also exact.  Since all the $\ZH$-modules $C_i \left ( \widetilde{Y}^{(n)} , \widetilde{X}^{(n)} \right ) $ are free, and application of Schanuel's Lemma to~\eqref{eq:CYX} and a trivial resolution of $C_0 \left ( \widetilde{Y}^{(n)} , \widetilde{X}^{(n)} \right )$ shows that  
$ Z_{n} \left ( \widetilde{Y}^{(n)},\widetilde{X}^{(n)} \right )$ is a stably free $\Z H$-module.

 Thus we have shown that the short exact sequence~\eqref{SES} satisfies the three hypothesis of Lemma~\ref{RetractionLemma}. Before invoking this lemma and concluding the proof, we set up notation for the norms required to specify representatives of $FV_G^{n+1}$ and $FV_H^{n+1}$.

Let $\|\cdot\|_1$ denote the $\ell_1$-norm on $C_{i}(\widetilde Y)$ induced by the basis consisting on all $i$-cells of $\widetilde Y$.  Let $\|\cdot\|_{Z_n(\widetilde Y)}$ denote the $\ell_1$-norm on $Z_n  (\widetilde Y)$ induced by a free $\ZG$-basis; by definition this is also filling norm on $Z_n  (\widetilde Y)$. Then (a representative of) $FV_G^{n+1}$ is given by 
\begin{equation} FV_G^{n+1}(k)=\max\left\{ \| \gamma \|_{Z_n(\widetilde Y)}  \colon \gamma\in Z_1(\widetilde Y),\ \| \gamma \|_1\leq k  \right\}.\end{equation}
Since $C_{n+1}(\widetilde X) \subseteq C_{n+1}(\widetilde Y)$ is a free factor,   the $\ell_1$-norm on $C_{n+1}(\widetilde X)$ induced by the $(n+1)$--cells of $\widetilde X$ equals the restriction of $\|\cdot\|_1$ to $C_{n+1}(\widetilde X)$. Let $\|\cdot\|_{Z_n(\widetilde X)}$ denote the filling-norm on $Z_n(\widetilde X)$ as a $\ZH$-module induced by the boundary map $C_{n+1}(\widetilde X) \overset{\partial_{n+1}}\to Z_n(\widetilde X)$. Then 
\begin{equation}FV_H^{n+1}(k)=\max\left\{ \| \gamma \|_{Z_n(\widetilde X)}  \colon \gamma\in Z_1(\widetilde X),\ \| \gamma \|_1\leq k  \right\}.\end{equation}

By Lemma~\ref{RetractionLemma} applied to the short exact sequence~\eqref{SES}, there exists a constant $C_1>0$ and a morphism  of $\ZH$-modules $\rho\colon Z_n(\widetilde Y) \to Z_n(\widetilde X)$ such that 
\begin{equation}\label{eq:main}
\| \rho (\alpha) \|_{Z_n(\widetilde X)} \leq C_1 \cdot \|\alpha\|_{Z_n(\widetilde Y)},
\end{equation}
for every $\alpha \in Z_n(\widetilde Y)$, and $\rho \circ \imath$ is the identity on $Z_n(\widetilde X)$.  

Let $k\in \N$ and let $\gamma\in Z_n(\widetilde X)$ such that 
$\|\gamma\|_1\leq k$.  Then~\eqref{eq:main} implies  that
\begin{equation}
\| \gamma \|_{Z_n(\widetilde X)} = \| \rho  \circ \iota (\gamma) \|_{Z_n(\widetilde X)}   \leq C \cdot \| \iota ( \gamma ) \|_{Z_n(\widetilde Y)}    \leq C  \cdot FV_{G}^{n+1} (k). \end{equation} 
Since $\gamma$ was arbitrary, we have $FV^{n+1}_{H}(k) \leq C  \cdot FV_{G}^{n+1} (k)$.
\end{proof}

\begin{remark} 
The proof of Theorem~\ref{Main} does not apply to obtain that $FV_H^{m+1}  \preceq  FV_G^{m+1}$ for $m<n$. As mentioned in the introduction, that statement is false.  The argument breaks down since $Z_m(\widetilde M_f)$ is not projective if $m<n$. 
\end{remark}

 \bibliographystyle{plain}

\end{document}